\documentclass[reqno,12pt,a4paper]{amsart}
\usepackage{amscd,amsfonts,amssymb,amsthm,latexsym}
\usepackage{mathtools}
\usepackage{bm}
\usepackage{srcltx}

\theoremstyle{plain}
\newtheorem{theorem}{Theorem}[section]
\newtheorem*{theorem*}{Theorem}

\newtheorem{lemma}{Lemma}[section]

\theoremstyle{definition}

\newtheorem*{definition*}{Definition}

\theoremstyle{remark}

\newtheorem*{remark*}{Remark}

\numberwithin{equation}{section}

\begin{document}
\raggedbottom %Ќужно, чтобы текст вертикально на раст€гивалс€

\title[Generalization of Romanoff's theorem]{Generalization of Romanoff's theorem}

\author{Artyom Radomskii}

\begin{abstract}We generalize Romanoff's theorem. Also, we obtain a result on sums related to Euler's totient function.
\end{abstract}

 \address{HSE University, Faculty of Computer Science, Pokrovsky Boulevard
 11, Moscow, 109028 Russia}


\email{artyom.radomskii@mail.ru}

\maketitle

\section{Introduction}

Let $\varphi$ denote Euler's totient function. We prove

\begin{theorem}\label{T1}
Let $a$, $d$, and $s$ be integers with $a>1$, $d\geq 1$, and $s\geq 1$. Let
\[
R(n)=b_{d} n^{d}+\ldots + b_0
\]be a polynomial with integer coefficients such that $b_d >0$ and $R: \mathbb{N}\to \mathbb{N}$. Then there is a constant $c=c(R, a, s)>0$, depending only on $R$, $a$, and $s$, such that
\begin{equation}\label{T1:ineq}
N \leq \sum_{n=1}^{N} \left(\frac{a^{R(n)}-1}{\varphi(a^{R(n)}-1)}\right)^{s} \leq c N
\end{equation} for any positive integer $N$.
\end{theorem}

\begin{theorem}\label{T2}
Let $a$ and $d$ be integers with $a>1$ and $d\geq 1$. Let
\[
R(n)=b_{d} n^{d}+\ldots + b_0
\]be a polynomial with integer coefficients such that $b_d >0$ and $R: \mathbb{N}\to \mathbb{N}$. Then there are positive constants $x_{0}=x_{0}(R,a)$, $c_{1}=c_{1}(R, a)$, and $c_{2}=c_{2}(R, a)$, depending only on $R$ and $a$, such that
\begin{align*}
c_{1}&\frac{x}{(\ln x)^{1-1/d}}\\
&\quad\leq \#\big\{1\leq n \leq x: \text{there are $p\in \mathbb{P}$ and $j\in \mathbb{N}$ such that $p + a^{R(j)}=n$}\big\}\\
 &\quad\leq c_{2}\frac{x}{(\ln x)^{1-1/d}}
\end{align*} for any real number $x\geq x_{0}$.
\end{theorem}
Theorem \ref{T2} extends a famous result of Romanoff \cite{Romanoff} which showed the same result but with $R(n)=n$, and a result of the author \cite[Theorem 1.9]{Rad.Izv} which showed the same result but with $R(n)=n^{d}$.

\section{Notation}

We reserve the letters $p$, $q$ for primes. In particular, the sum $\sum_{p\leq K}$ should be interpreted as being over all prime numbers not exceeding $K$. By $\pi (x)$ we denote the number of primes not exceeding $x$. Let $\# A$ denote the number of elements of a finite set $A$. By $\mathbb{Z}$ and $\mathbb{N}$ we denote the sets of all integers and positive integers respectively. By $\mathbb{P}$ we denote the set of all prime numbers. Let $(a_1,\ldots, a_n)$ be the greatest common divisor of integers $a_1,\ldots, a_n$. Let $\varphi$ denote Euler's totient function, i.\,e.
\[
\varphi(n)=\#\{1\leq m \leq n: (m,n)=1\},\quad n\in \mathbb{N}.
\]

By definition, we put
\[
\sum_{\varnothing} = 0,\qquad \prod_{\varnothing}=1.
\]The symbol $b|a$ means that $b$ divides $a$. For fixed $a$ the sum $\sum_{b|a}$ and the product $\prod_{b|a}$ should be interpreted as being over all positive divisors of $a$. If $x$ is a real number, then $[x]$ denotes its integral part, $\lceil x\rceil$ is the smallest integer $n$ such that $n\geq x$, and $\{x\}:= x - [x]$ is a fractional part of $x$. We put $\log_{a}x:=\ln x/\ln a$.

For a prime $p$ and an integer $a$ with $(a,p)=1$, let $h_{a}(p)$ denote the order of $a$ modulo $p$, which is to say that $h_{a}(p)$ is the least positive integer $h$ such that $a^{h}\equiv 1$ (mod $p$). We observe that $a^{p-1}\equiv 1$ (mod $p$) (Fermat's theorem), and hence $h_{a}(p)$ exists and $1\leq h_{a}(p) \leq p-1$.

\section{Preparatory Lemmas}

\begin{lemma}\label{L1}
Let $d$ be an integer with $d\geq 1$. Let
\[
R(n)=b_{d}n^{d}+\ldots + b_{0}
\] be a polynomial with integer coefficients such that $b_{d}>0$ and $R: \mathbb{N}\to \mathbb{N}$. Then there are positive constants $c_1=c_{1}(R)$ and $c_2= c_{2}(R)$, depending only on $R$, such that
\[
c_{1} n^{d}\leq R(n)\leq c_{2}n^{d}
\]for any positive integer $n$.
\end{lemma}
\begin{proof} There is a positive integer $N_{0}=N_{0} (b_d,\ldots, b_0)=N_{0}(R)$  such that
\[
\frac{b_{d}n^{d}}{2}\leq b_{d}n^{d}+\ldots + b_0 \leq 2 b_{d} n^{d}
\]for any positive integer $n\geq N_{0}$. We put
\[
M=\max_{1\leq n \leq N_{0}(R)} R(n),\qquad m= \min_{1\leq n \leq N_{0}(R)}R(n),
\]and
\[
c_2= \max (M, 2 b_d),\qquad c_1=\min \left(\frac{m}{(N_{0}(R))^{d}}, \frac{b_d}{2}\right).
\]We see that $c_1= c_{1}(R)$ and $c_2=c_{2}(R)$ are positive constants, depending only on $R$.

Let us show that
\[
c_1 n^{d}\leq R(n)\leq c_{2}n^{d}
\] for any positive integer $n$. Indeed, if $n> N_{0}(R)$, then
\begin{align*}
R(n)&\leq 2b_{d}n^{d}\leq c_{2} n^{d},\\
R(n)&\geq \frac{b_{d}n^{d}}{2}\geq c_{1} n^{d}.
\end{align*}If $1 \leq n \leq N_{0}(R)$, then
\[
R(n)\leq M \leq c_{2}\leq c_{2} n^d,
\]and
\[
R(n) \geq m = \frac{m}{(N_{0}(R))^{d}} (N_{0}(R))^{d}\geq c_{1} (N_{0}(R))^{d}\geq c_{1} n^{d}.
\]Lemma \ref{L1} is proved.

\end{proof}

\begin{lemma}\label{L8}
Let $a\in \mathbb{Z}$, $b\in \mathbb{N}$, $N\in \mathbb{N}$. Then
\begin{equation}\label{L8:ineq}
\#\{t\in \mathbb{Z}: 1\leq a+bt\leq N\}\leq \frac{N}{b}+1.
\end{equation}
\end{lemma}
\begin{proof}We have
\begin{align*}
\Omega :&=\{t\in \mathbb{Z}: 1\leq a+bt\leq N\}=\{t\in \mathbb{Z}: 0< a+bt\leq N\}\\
&=\{t\in \mathbb{Z}: [-a/b]+1\leq t\leq [(N-a)/b]\}.
\end{align*} We see that if $[-a/b]+1 > [(N-a)/b]$, then $\Omega = \emptyset$, $\#\Omega = 0$, and \eqref{L8:ineq} holds.

 Suppose that $[-a/b]+1 \leq [(N-a)/b]$. Then we have
 \begin{align*}
 \#\Omega &= \left[\frac{N-a}{b}\right] - \left[\frac{-a}{b}\right] = \frac{N-a}{b} - \left\{\frac{N-a}{b}\right\}-
 \left(\frac{-a}{b}-\left\{\frac{-a}{b}\right\}\right)\\
 &=\frac{N}{b} + \left\{\frac{-a}{b}\right\} - \left\{\frac{N-a}{b}\right\}\leq \frac{N}{b} + 1.
 \end{align*}Lemma \ref{L8} is proved.
\end{proof}

\begin{lemma}\label{L2}
Let $a$ be an integer and $p$ be a prime with $(a,p)=1$. Let $d$ be a positive integer. Then $a^{d}\equiv 1$ \textup{(mod $p$)} if and only if $d\equiv 0$ \textup{(mod $h_{a}(p)$)}.
\end{lemma}
\begin{proof} See, for example, \cite[Theorem 88]{Hardy_Wright}.
\end{proof}

\begin{lemma}\label{L3}
Let $d$ and $m$ be positive integers. Let
\[
f(x)=\sum_{i=0}^{d} b_{i} x^{i},
\]where $b_{0},\ldots, b_d$ are integers with $(b_{0},\ldots, b_d, m)=1$. Let $\rho (f,m)$ denote the number of solutions of the congruence $f(x)\equiv 0$ \textup{(mod $m$)}. Then
\[
\rho (f, m)\leq c d m^{1-1/d},
\]where $c>0$ is an absolute constant.
\end{lemma}
\begin{proof}
This is \cite[Theorem 2]{Konyagin}.
\end{proof}

\begin{lemma}\label{L_gcd}
Let $d$ and $m$ be positive integers. Let
\[
f(x)=\sum_{i=0}^{d} b_{i} x^{i},
\]where $b_{0},\ldots, b_d$ are integers, $b_{d}>0$. Let $\rho (f,m)$ denote the number of solutions of the congruence $f(x)\equiv 0$ \textup{(mod $m$)}, and let $\delta= (b_{0},\ldots, b_{d}, m)$. Then
\begin{equation}\label{L_gcd:INEQ_1}
\rho (f, m)\leq c d {\delta}^{1/d} m^{1-1/d},
\end{equation}where $c>0$ is an absolute constant. In particular, we have
\begin{equation}\label{L_gcd:INEQ_2}
\rho (f, m)\leq c d b_{d}^{1/d} m^{1-1/d}.
\end{equation}
\end{lemma}
\begin{proof}
Let $\delta =1$. By Lemma \ref{L3}, we have
\begin{equation}\label{L_gcd:i1}
\#\{1\leq n\leq m: f(n)\equiv 0\text{ (mod $m$)}\}\leq c d m^{1-1/d},
\end{equation}where $c>0$ is an absolute constant.

Suppose that $\delta>1$. We put
\begin{gather*}
\widetilde{m}=\frac{m}{\delta},\qquad \widetilde{f}(x)=\frac{f(x)}{\delta},\\
U=\big\{1\leq j \leq \widetilde{m}: \widetilde{f}(j)\equiv 0\text{ (mod $\widetilde{m}$)}\big\}.
\end{gather*}Since
\[
\bigg(\frac{b_{0}}{\delta},\ldots, \frac{b_{d}}{\delta}, \frac{m}{\delta}\bigg)=1,
\]by Lemma \ref{L3} we have
\begin{equation}\label{L_gcd:U}
\#U \leq c d \widetilde{m}^{1-1/d}.
\end{equation}

Let us show that
\begin{equation}\label{L_gcd:BASIC}
\#\{ 1\leq n \leq m: f(n)\equiv 0\text{ (mod $m$)}\}= \sum_{j\in U}\#\{t\in \mathbb{Z}:
1\leq j+ \widetilde{m}t\leq m\}.
\end{equation}We put
\[
\Omega = \{ 1\leq n \leq m: f(n)\equiv 0\text{ (mod $m$)}\},\quad A=\#\Omega,
\]and
\[
B=\sum_{j\in U}\#\{t\in \mathbb{Z}:
1\leq j+ \widetilde{m}t\leq m\}.
\]Let $n\in \Omega$. Then there is an integer $l$ such that $f(n)=ml$. Hence, $\widetilde{f}(n)= \widetilde{m}l$.
 Since $\widetilde{f}(n)\equiv 0$ (mod $\widetilde{m}$), there are $j\in U $ and $t\in \mathbb{Z}$ such that $n=j+ \widetilde{m}t$. Since $1\leq n \leq m$, we obtain $1\leq j+ \widetilde{m}t\leq m$. We see that $A\leq B$.

 Let $j\in U$ and let $t$ be an integer such that $1\leq j+ \widetilde{m}t\leq m$. We put $n= j+ \widetilde{m}t$. Hence, $n$ is an integer and $1\leq n \leq m$. Since $n \equiv j$ (mod $\widetilde{m}$), we have $\widetilde{f}(n)\equiv \widetilde{f}(j)$ (mod $\widetilde{m}$), and hence $\widetilde{f}(n)\equiv 0$ (mod $\widetilde{m}$). Therefore, there is an integer $l$ such that $\widetilde{f}(n)=\widetilde{m}l$. We obtain
 \[
 f(n)=\delta \widetilde{f}(n)= \delta \widetilde{m}l= ml,
 \]and hence $f(n)\equiv 0$ (mod $m$). Therefore, $n\in \Omega$. We see that $B\leq A$. Hence, $A=B$ and \eqref{L_gcd:BASIC} is proved.

 Let $j\in U$. By Lemma \ref{L8}, we have
 \[
 \#\{t\in \mathbb{Z}:
1\leq j+ \widetilde{m}t\leq m\}\leq \frac{m}{\widetilde{m}}+1= \delta+1\leq 2\delta .
 \]Applying \eqref{L_gcd:U} and \eqref{L_gcd:BASIC}, we obtain
 \begin{align}
 \#\{ 1\leq n \leq m: f(n)\equiv 0\text{ (mod $m$)}\} \leq 2\delta \#U&\leq 2 \delta  c d \left(\frac{m}{\delta}\right)^{1-1/d}\notag\\
 &= 2 c d \delta^{1/d} m^{1-1/d}.\label{L_gcd:i2}
 \end{align}

 From \eqref{L_gcd:i1} and \eqref{L_gcd:i2} we obtain
 \[
 \#\{ 1\leq n \leq m: f(n)\equiv 0\text{ (mod $m$)}\}\leq
  c_{1} d \delta^{1/d} m^{1-1/d},
 \]where $c_{1}=2c$ is a positive absolute constant. Therefore, \eqref{L_gcd:INEQ_1} is proved. Since $\delta>0$, $b_{d}>0$, and $\delta| b_{d}$, we have $\delta \leq b_{d}$, and \eqref{L_gcd:INEQ_2} follows from \eqref{L_gcd:INEQ_1}. Lemma \ref{L_gcd} is proved.
\end{proof}

\begin{lemma}\label{L4}
Let $n$ and $s$ be integers with $n\geq 3$ and $s\geq 1$. Then
\[
\sum_{p|n} \frac{(\ln p)^{s}}{p} \leq c s^{s} (\ln\ln n)^{s},
\]where $c>0$ is an absolute constant.
\end{lemma}
\begin{proof}
We have
\[
S=\sum_{p|n} \frac{(\ln p)^{s}}{p} = \sum_{\substack{p|n\\ p\leq (\ln n)^{s}}} \frac{(\ln p)^{s}}{p} +
\sum_{\substack{p|n\\ p> (\ln n)^{s}}} \frac{(\ln p)^{s}}{p}= S_{1}+ S_{2}.
\]First we estimate $S_{2}$. We have
\[
S_{2}\leq \frac{1}{(\ln n)^{s}}\sum_{p|n}(\ln p)^{s}.
\]Let
\[
n=p_{1}^{\alpha_{1}}\cdots p_{r}^{\alpha_{r}}
\]be the unique factorization of $n$ into powers of distinct primes. Then
\begin{align*}
\sum_{p|n}(\ln p)^{s}&= (\ln p_{1})^{s}+\ldots + (\ln p_{r})^{s}\leq (\ln p_1+\ldots+ \ln p_{r})^{s}\\
 &=(\ln (p_1\cdots p_r))^{s}\leq (\ln n)^{s}.
\end{align*} We obtain $S_{2}\leq 1$. Since $n\geq 3$, we have $1\leq 11 \ln\ln n$. Hence,
\[
S_{2}\leq 1\leq (11\ln\ln n)^{s}\leq c_{0} s^{s}(\ln\ln n)^{s},
\]where $c_{0}>0$ is an absolute constant.

Now we estimate $S_{1}$. It can be shown (see, for example, \cite[Lemma 3.5]{Rad.Izv}) that there is an absolute constant $c_{1}>0$ such that
\[
\sum_{p\leq x} \frac{(\ln p)^{s}}{p}\leq c_{1}(\ln x)^{s}
\]for any real number $x\geq 1$. We obtain
\[
S_{1}\leq  \sum_{p\leq (\ln n)^{s}} \frac{(\ln p)^{s}}{p}\leq c_{1} s^{s} (\ln\ln n)^{s}.
\]We see that $S\leq c s^{s}(\ln\ln n)^{s}$, where $c=c_{0}+c_{1}$ is a positive absolute constant. Lemma \ref{L4} is proved.
\end{proof}

\begin{lemma}\label{L5}
Let $a$, $d$, and $s$ be integers such that $a>1$, $d\geq 1$, and $s\geq 1$. Then
\[
\sum_{\substack{p:\\ (p,a)=1}} \frac{(\ln p)^{s}}{p (h_{a}(p))^{1/d}}\leq c(a,d,s),
\]where $c(a,d,s)>0$ is a constant, depending only on $a$, $d$, and $s$.
\end{lemma}
\begin{proof} We put
\[
\gamma_{n}= \sum_{\substack{p:\\ (p,a)=1\\ h_{a}(p)=n}} \frac{(\ln p)^{s}}{p},\quad n=1, 2, \ldots,
\]and
\[
G(z)=\sum_{n \leq z} \gamma_{n},\qquad P(z)=\prod_{n\leq z} (a^{n}-1).
\] Let $z\geq 10$. It is easy to see that $P(z)\geq 3$ and
\[
P(z)\leq \prod_{n\leq z} a^{n}= a^{1+2+\ldots+ [z]}\leq a^{z^{2}}.
\]Applying Lemma \ref{L4}, we obtain
\[
G(z)=\sum_{n\leq z} \sum_{\substack{p:\\ (p,a)=1\\ h_{a}(p)=n}} \frac{(\ln p)^{s}}{p}\leq
 \sum_{p| P(z)} \frac{(\ln p)^{s}}{p}\leq c s^{s} (\ln\ln P(z))^{s}\leq c_{1}(a,s) (\ln z)^{s},
\]where $c_{1}(a,s)>0$ is a constant, depending only on $a$ and $s$. If $1 \leq z < 10$, then we have
\[
G(z)\leq G(10)\leq c_{1}(a,s) (\ln 10)^{s}\leq c_{1}(a,s) (\ln 10)^{s} (\ln (z+2))^{s} = c_{2}(a,s) (\ln (z+2))^{s}
\](we used here that $1\leq \ln (z+2)$ for any $z\geq 1$). Hence, there is a constant $c(a,s)>0$, depending only on $a$ and $s$, such that
\[
0\leq G(z) \leq c(a,s)(\ln (z+2))^{s}
\] for any real number $z\geq 1$.

Applying partial summation (see, for example, \cite[Theorem 2.1.1]{Murty}), for any $z\geq 1$, we have
\[
\sum_{n\leq z} \frac{\gamma_{n}}{n^{1/d}} = \frac{G(z)}{z^{1/d}}+ \frac{1}{d}\int_{1}^{z}\frac{G(t)}{t^{1+1/d}}\,dt.
\] Since $G(z) z^{-1/d}\to 0$ as $z\to +\infty$, we obtain
\[
\sum_{n=1}^{\infty} \frac{\gamma_{n}}{n^{1/d}}=\frac{1}{d}\int_{1}^{\infty}\frac{G(t)}{t^{1+1/d}}\,dt\leq
\frac{c(a,s)}{d}\int_{1}^{\infty}\frac{(\ln (t+2))^{s}}{t^{1+1/d}}\,dt= c(a,d,s),
\]where $c(a,d,s)>0$ is a constant, depending only on $a$, $d$, and $s$. In particular, we see that the series $\sum_{n\geq 1} \gamma_{n} n^{-1/d}$ is convergent.

We have
\begin{align*}
\sum_{n=1}^{\infty} \frac{\gamma_{n}}{n^{1/d}}&= \sum_{n\geq 1} \sum_{\substack{p:\\ (p,a)=1\\ h_{a}(p)=n}} \frac{(\ln p)^{s}}{p(h_{a}(p))^{1/d}}= \sum_{\substack{p:\\ (p,a)=1}} \sum_{\substack{n\geq 1:\\ h_{a}(p)=n}}\frac{(\ln p)^{s}}{p(h_{a}(p))^{1/d}}\\
&= \sum_{\substack{p:\\ (p,a)=1}}\frac{(\ln p)^{s}}{p(h_{a}(p))^{1/d}} \sum_{\substack{n\geq 1:\\ h_{a}(p)=n}}1
= \sum_{\substack{p:\\ (p,a)=1}}\frac{(\ln p)^{s}}{p(h_{a}(p))^{1/d}}.
\end{align*}This gives the result.
\end{proof}

\begin{lemma}\label{L6}
Let $\alpha$ be a real number with $0<\alpha <1$. Then there is a constant $C(\alpha)>0$, depending only on $\alpha$, such that the following holds. Let $M$ be a real number with $M\geq 1$, let $a_1,\ldots, a_{N}$ be positive integers \textup{(}not necessarily distinct\textup{)} with $a_{n}\leq M$ for all $1\leq n \leq N$. We define
\[
\omega(d)=\#\{1\leq n\leq N: a_{n}\equiv 0\ \textup{\text{(mod $d$)}}\}
\]for any positive integer $d$. Let $s$ be a positive integer. Then
\[
\sum_{n=1}^{N}\left(\frac{a_{n}}{\varphi(a_{n})}\right)^{s}\leq (C(\alpha))^{s}
\biggl(N+\sum_{p\leq (\ln M)^{\alpha}} \frac{\omega(p)(\ln p)^{s}}{p}\biggr).
\]
\end{lemma}

\begin{proof} This is \cite[Theorem 1.1]{Rad.Izv}.
\end{proof}

\begin{lemma}\label{L7}
Let $A=\{a_{n}\}_{n=1}^{\infty}$ be a sequence of positive integers \textup{(}not necessarily distinct\textup{)}. We put
\begin{align*}
N_{A}(x)&=\#\{n\in \mathbb{N}: a_n \leq x\},\\
\textup{ord}_{A}(n)&=\#\{j\in \mathbb{N}: a_j=n\},\qquad  n\in\mathbb{N},\\
\rho_{A}(x)&=\max_{n\leq x}\textup{ord}_{A}(n).
\end{align*}Suppose that $\textup{ord}_{A}(n)<+\infty$ for any positive integer $n$. Suppose that there are constants $\gamma_{1}>0$, $\gamma_{2}>0$, $\alpha>0$, $x_{0}\geq 10$ such that the following holds.  For any real number $x\geq x_0$, we have
\begin{gather}
N_{A}(x)>0,\label{L:R2_i1}\\
N_{A}\biggl(\frac{x}{2}\biggr)\geq \gamma_1 N_{A}(x),\label{L:R2_i2}\\
\sum_{\substack{k\in \mathbb{N}:\\ a_k < x}}\sum_{p\leq (\ln x)^{\alpha}}
\frac{\#\{n\in\mathbb{N}: a_{k}< a_n \leq x\ \text{and } a_n\equiv a_k\ \textup{(mod $p$)}\}\ln p}{p}\leq \gamma_2 (N_{A}(x))^2.\label{L:R2_i3}
\end{gather}Then there is a positive constant $c=c(\gamma_1, \gamma_2, \alpha)$, depending only on $\gamma_1$, $\gamma_2$, $\alpha$, such that
\begin{align*}
\#\{1&\leq n \leq x: \text{there are $p\in \mathbb{P}$ and $j\in \mathbb{N}$ such that $p+a_j=n$}\}\\
&\geq c x\frac{N_{A}(x)}{N_{A}(x)+\rho_{A}(x)\ln x}
\end{align*}for any real number $x\geq x_0$.
\end{lemma}

\begin{proof}
This is \cite[Theorem 1.6]{Rad.Izv}.
\end{proof}

\section{Proof of theorems \ref{T1} and \ref{T2}}

\begin{proof}[Proof of Theorem \ref{T1}.] We observe that $1\leq \varphi (m)\leq m$ for any positive integer $m$. It is clear that $a^{R(n)}-1$ is a positive integer for any $n\in\mathbb{N}$. Hence,
\[
\left(\frac{a^{R(n)}-1}{\varphi(a^{R(n)}-1)}\right)^{s} \geq 1
\]for any $n\in \mathbb{N}$. Fix a positive integer $N$ and put
\[
S= \sum_{n=1}^{N} \left(\frac{a^{R(n)}-1}{\varphi(a^{R(n)}-1)}\right)^{s}.
\]We obtain $S\geq N$.

Now we prove the second inequality in \eqref{T1:ineq}. By Lemma \ref{L1}, we have
\[
R(n)\leq c(R)n^{d}
\] for any $n\in \mathbb{N}$, where $c(R)>0$ is a constant depending only on $R$. Applying Lemma \ref{L6} with $a_n= a^{R(n)}-1$ ($n=1,\ldots, N$), $\alpha = 1/ (2d)$, $M=a^{c(R)N^{d}}$, we obtain
\begin{equation}\label{T8:i1}
S\leq (C(d))^{s}\bigg(N+ \sum_{p\leq c(R,a)N^{1/2}} \frac{\omega(p) (\ln p)^{s}}{p}\bigg),
\end{equation}where $c(R,a)>0$ is a constant, depending only on $R$ and $a$, and
\[
\omega(p)=\#\{1\leq n \leq N: a^{R(n)}\equiv 1 \text{ (mod $p$)}\}.
\] We have
\begin{equation}\label{T8:i2}
\sum_{p\leq c(R,a)N^{1/2}} \frac{\omega(p) (\ln p)^{s}}{p}=
\sum_{\substack{p\leq c(R,a)N^{1/2}\\p|a}}  +
\sum_{\substack{p\leq c(R,a)N^{1/2}\\(p,a)=1}} = S_{1}+ S_{2}.
\end{equation}

Let $p$ be a prime with $p|a$. Let us show that $\omega(p)=0$. Indeed, assume the converse. Then there are positive integers $n$, $b$, and $c$ with $1\leq n \leq N$ such that $(pb)^{R(n)}- 1 = cp$, but this is impossible. We obtain
\begin{equation}\label{T8:i3}
S_{1} = 0.
\end{equation}

Let $p$ be in the range of summation of $S_{2}$. By Lemma \ref{L2}, we have
\[
\omega(p)=\#\{1\leq n \leq N: R(n)\equiv 0\text{ (mod $h_{a}(p))$}\}.
\]Let
\[
U=\{n\in\{0,\ldots, h_{a}(p)-1\}: R(n)\equiv 0\text{ (mod $h_{a}(p))$}\}.
\]From \eqref{L_gcd:INEQ_2} we obtain
\[
\# U \leq c_{0} d b_{d}^{1/d} (h_{a}(p))^{1-1/d},
\]where $c_0 >0$ is an absolute constant. It is clear that
\[
\omega (p)= \sum_{j\in U}\#\{t\in \mathbb{Z}: 1\leq j+ h_{a}(p)t\leq N\}.
\] Let $j\in U$. By Lemma \ref{L8}, we have
\[
\#\{t\in \mathbb{Z}: 1\leq j+ h_{a}(p)t\leq N\} \leq \frac{N}{h_{a}(p)}+ 1.
\]We see that
\[
h_{a}(p)\leq p-1 < p \leq c(R,a) N^{1/2}\leq c(R,a) N,
\]and hence
\[
\#\{t\in \mathbb{Z}: 1\leq j+ h_{a}(p)t\leq N\} \leq (1+ c(R,a))\frac{N}{h_{a}(p)}=c_{1}(R,a)\frac{N}{h_{a}(p)}.
\]We obtain
\[
\omega(p) \leq c_{1}(R,a)\frac{N}{h_{a}(p)} \#U\leq
c_{1}(R,a)\frac{N}{h_{a}(p)} c_{0} d b_{d}^{1/d} (h_{a}(p))^{1-1/d} = c_{2}(R,a) \frac{N}{(h_{a}(p))^{1/d}},
\]where $c_{2}(R,a)>0$ is a constant, depending only on $R$ and $a$.

We obtain (see also Lemma \ref{L5})
\begin{align}
S_{2}&= \sum_{\substack{p\leq c(R,a)N^{1/2}\\(p,a)=1}}\frac{\omega(p) (\ln p)^{s}}{p} \leq c_{2}(R,a) N \sum_{\substack{p\leq c(R,a)N^{1/2}\\(p,a)=1}} \frac{(\ln p)^{s}}{p(h_{a}(p))^{1/d}} \notag\\
&\leq c_{2}(R,a) N \sum_{\substack{p:\\(p,a)=1}} \frac{(\ln p)^{s}}{p(h_{a}(p))^{1/d}}\leq c_{3}(R,a,s) N,\label{T8:i4}
\end{align}where $c_{3}(R,a,s)>0$ is a constant, depending only on $R$, $a$, and $s$.

From \eqref{T8:i1}, \eqref{T8:i2}, \eqref{T8:i3}, and \eqref{T8:i4} we obtain
\[
S\leq (C(d))^{s} (1 + c_{3}(R,a,s))N = c(R,a,s)N,
\]where $c(R,a,s)>0$ is a constant, depending only on $R$, $a$, and $s$. Theorem \ref{T1} is proved.
\end{proof}

\begin{proof}[Proof of Theorem \ref{T2}.]We put $a_{n}=a^{R(n)},$ $n=1, 2, \ldots .$ We are going to apply Lemma \ref{L7} with $A=\{a_n: n=1, 2,\ldots\}$. We keep the notation of Lemma \ref{L7}. We assume that $x\geq x_{0}(R,a)$, where $x_{0}(R,a)>0$ is a constant depending only on $R$ and $a$; we will choose the constant $x_{0}(R,a)$ later, and it will be large enough.

For any positive integer $n$, we have
\[
\textup{ord}_{A}(n)=\#\{j\in \mathbb{N}: a^{R(j)}=n\}=\#\{j\in \mathbb{N}: R(j)=\log_{a}n\}\leq d.
\]It is clear that
\[
\textup{ord}_{A}(a^{R(1)})\geq 1.
\]We may assume that $x_{0}(R,a)\geq a^{R(1)}$, and hence, if $x\geq x_{0}(R,a)$, then
\[
\rho_{A}(x)=\max_{n\leq x}\textup{ord}_{A}(n)\geq \textup{ord}_{A}(a^{R(1)})\geq 1.
\]We obtain
\begin{equation}\label{T2:rho}
1 \leq \rho_{A}(x) \leq d
\end{equation}for any real number $x \geq x_{0}(R,a)$.

By Lemma \ref{L1}, there are positive constants $c_{1}(R)$ and $c_{2}(R)$, depending only on $R$, such that
\begin{equation}\label{T2:R_est}
c_{1}(R)n^{d} \leq R(n) \leq c_{2}(R) n^{d}
\end{equation} for any positive integer $n$. We recall that
\[
N_{A}(x)=\#\{n\in \mathbb{N}: a^{R(n)}\leq x\}.
\] We see that there are positive constants $c_{3}(R,a)$ and $c_{4}(R,a)$ such that
\begin{equation}\label{T2:N_A_est}
c_{3}(R,a) (\ln x)^{1/d} \leq N_{A}(x)\leq c_{4}(R,a) (\ln x)^{1/d}
\end{equation} for any real number $x \geq x_{0}(R,a)$, if $x_{0}(R,a)$ is large enough. Hence, there is a positive constant $\gamma_1 = \gamma_1 (R,a)$, depending only on $R$ and $a$, such that
\[
N_{A}(x)>0
\]and
\[
N_{A}\bigg(\frac{x}{2}\bigg) \geq \gamma_{1} N_{A}(x)
\]for any real number $x \geq x_{0}(R,a)$. Thus, \eqref{L:R2_i1} and \eqref{L:R2_i2} hold.

Now we check \eqref{L:R2_i3}. We take $\alpha = 1/(2d)$. Fix a positive integer $k$ such that $a^{R(k)}< x$, and put
\[
S=\sum_{p\leq (\ln x)^{1/(2d)}} \frac{\lambda (k,p)\ln p}{p},
\]where
\[
\lambda(k,p):= \#\{n\in \mathbb{N}: a^{R(k)}< a^{R(n)}\leq x\text{ and } a^{R(n)}\equiv a^{R(k)}
\text{ (mod $p$)}\}.
\]We have
\begin{equation}\label{T2:S_rep}
S=\sum_{\substack{p\leq (\ln x)^{1/(2d)}\\ p|a}} \frac{\lambda (k,p)\ln p}{p}+
\sum_{\substack{p\leq (\ln x)^{1/(2d)}\\ (p,a)=1}} \frac{\lambda (k,p)\ln p}{p}=S_{1}+S_{2}.
\end{equation}

Since $\lambda (k,p)\leq N_{A}(x)$, we obtain
\begin{equation}\label{T2:S1_est}
S_{1} \leq N_{A}(x)\sum_{p|a} \frac{\ln p}{p}= c(a)N_{A}(x).
\end{equation}

Let $p$ be in the range of summation of $S_{2}$. Taking into account that $(a, p)=1$ and applying Lemma \ref{L2}, we obtain
\begin{align*}
\lambda (k,p)&=\#\left\{n\in \mathbb{N}: R(k)<R(n)\leq \frac{\ln x}{\ln a}\text{ and }a^{R(n)-R(k)}\equiv 1\text{ (mod $p$)}\right\}\\
&=\#\left\{n\in \mathbb{N}: R(k)<R(n)\leq \frac{\ln x}{\ln a}\text{ and }R(n)-R(k)\equiv 0\text{ (mod $h_{a}(p)$)}\right\}.
\end{align*}By \eqref{T2:R_est}, we have
\begin{align*}
\lambda (k,p)&\leq \#\left\{n\in \mathbb{N}: c_{1}(R)n^{d}\leq \frac{\ln x}{\ln a}\text{ and }
R(n)-R(k)\equiv 0\text{ (mod $h_{a}(p)$)}\right\}\\
&=\#\bigg\{1\leq n \leq \bigg[\left(\frac{\ln x}{c_{1}(R)\ln a}\right)^{1/d}\bigg]:
R(n)-R(k)\equiv 0\text{ (mod $h_{a}(p)$)}\bigg\}.
\end{align*}

Let
\[
U=\{n\in \{0,\ldots, h_{a}(p)-1\}: R(n)-R(k)\equiv 0\text{ (mod $h_{a}(p)$)}\}.
\] By \eqref{L_gcd:INEQ_2}, we have
\begin{equation}\label{T2:U_est}
\#U\leq c_{0}d b_{d}^{1/d} (h_{a}(p))^{1-1/d},
\end{equation}where $c_{0}>0$ is an absolute constant.

 We have
\[
h_{a}(p)\leq p-1 \leq (\ln x)^{1/(2d)}\leq \left(\frac{\ln x}{c_{1}(R)\ln a}\right)^{1/d} -1\leq
\bigg[\left(\frac{\ln x}{c_{1}(R)\ln a}\right)^{1/d}\bigg],
\]if $x_{0}(R,a)$ is large enough.

 Let $j\in U$. Applying Lemma \ref{L8}, we obtain
 \begin{align*}
 \#\bigg\{t\in \mathbb{Z}: 1\leq j+h_{a}(p)t\leq \bigg[\left(\frac{\ln x}{c_{1}(R)\ln a}\right)^{1/d}\bigg]\bigg\}&\leq
  \bigg[\left(\frac{\ln x}{c_{1}(R)\ln a}\right)^{1/d}\bigg] \bigg/ h_{a}(p)+1\\
  &\leq 2 \bigg[\left(\frac{\ln x}{c_{1}(R)\ln a}\right)^{1/d}\bigg] \bigg/ h_{a}(p).
   \end{align*}

   Hence (see \eqref{T2:U_est})
\begin{align*}
\lambda (k,p) &\leq \sum_{j\in U}\#\bigg\{t\in \mathbb{Z}: 1\leq j+h_{a}(p)t\leq \bigg[\left(\frac{\ln x}{c_{1}(R)\ln a}\right)^{1/d}\bigg]\bigg\}\\
&\leq
\bigg(2 \bigg[\left(\frac{\ln x}{c_{1}(R)\ln a}\right)^{1/d}\bigg] \bigg/ h_{a}(p)\bigg)\#U\leq c_{5}(R,a) \frac{(\ln x)^{1/d}}{(h_{a}(p))^{1/d}},
\end{align*}where $c_{5}(R,a)>0$ is a constant, depending only on $R$ and $a$.

Applying Lemma \ref{L5} and \eqref{T2:N_A_est}, we obtain
\begin{align}
S_{2}&=\sum_{\substack{p\leq (\ln x)^{1/(2d)}\\ (p,a)=1}} \frac{\lambda (k,p)\ln p}{p}\leq
c_{5}(R,a)(\ln x)^{1/d} \sum_{\substack{p\leq (\ln x)^{1/(2d)}\\ (p,a)=1}} \frac{\ln p}{p(h_{a}(p))^{1/d}}\notag\\
&\leq c_{5}(R,a)(\ln x)^{1/d} \sum_{\substack{p:\\ (p,a)=1}} \frac{\ln p}{p(h_{a}(p))^{1/d}}\leq
c_{6}(R,a)(\ln x)^{1/d}\leq c_{7}(R,a) N_{A}(x),\label{T2:S2_est}
\end{align}where $c_{7}(R,a)>0$ is a constant depending only on $R$ and $a$.

From \eqref{T2:S_rep}, \eqref{T2:S1_est} and \eqref{T2:S2_est}, we obtain
\[
\sum_{p\leq (\ln x)^{1/(2d)}} \frac{\lambda (k,p)\ln p}{p} \leq c_{8}(R,a) N_{A}(x)
\] for any positive integer $k$ such that $a^{R(k)}< x$. Hence,
\[
\sum_{\substack{k\in \mathbb{N}:\\ a^{R(k)}< x}}\sum_{p\leq (\ln x)^{1/(2d)}} \frac{\lambda (k,p)\ln p}{p}\leq
c_{8}(R,a) N_{A}(x) \sum_{\substack{k\in \mathbb{N}:\\ a^{R(k)}< x}} 1\leq c_{8}(R,a) (N_{A}(x))^{2},
\]where $c_{8}(R,a)$ is a positive constant depending only on $R$ and $a$. Thus, \eqref{L:R2_i3} holds with $\gamma_{2}=c_{8}(R,a)$.

By Lemma \ref{L7},
\begin{align*}
\#\{1&\leq n \leq x: \text{there are $p\in \mathbb{P}$ and $j\in \mathbb{N}$ such that $p+a^{R(j)}=n$}\}\\
&\geq c x\frac{N_{A}(x)}{N_{A}(x)+\rho_{A}(x)
\ln x},
\end{align*} where $c=c(R,a)$ is a positive constant, depending only on $R$ and $a$.

From \eqref{T2:rho} and \eqref{T2:N_A_est} we obtain
\begin{align*}
N_{A}(x) + \rho_{A}(x) \ln x&\leq c_{4}(R,a) (\ln x)^{1/d} + d\ln x\\
&\leq
(c_{4}(R,a) +d)\ln x = c_{9}(R,a)\ln x
\end{align*}and
\[
\frac{N_{A}(x)}{N_{A}(x)+\rho_{A}(x)
\ln x}\geq \frac{c_{3}(R,a) (\ln x)^{1/d}}{c_{9}(R,a)\ln x}=
\frac{c_{10}(R,a)}{(\ln x)^{1-1/d}}.
\]Hence,
\begin{align*}
\#\{1&\leq n \leq x: \text{there are $p\in \mathbb{P}$ and $j\in \mathbb{N}$ such that $p+a^{R(j)}=n$}\}\\
&\geq c_{1} \frac{x}{(\ln x)^{1-1/d}},
\end{align*} where $c_{1}=c_{1}(R,a)$ is a positive constant, depending only on $R$ and $a$.

By Chebyshev's theorem, $\pi(x) \leq c_0 x/ \ln x$, where $c_{0}>0$ is an absolute constant. Applying \eqref{T2:N_A_est}, we obtain
\begin{align*}
\#\{1&\leq n \leq x: \text{there are $p\in \mathbb{P}$ and $j\in \mathbb{N}$ such that $p+a^{R(j)}=n$}\}\\
&\leq \#\{p \leq x\}\cdot \#\{j\in \mathbb{N}:\ a^{R(j)}\leq x\} = \pi(x) N_{A}(x)\leq
c_0 \frac{x}{\ln x} c_{4}(R,a) (\ln x)^{1/d}\\
&= c_{2} \frac{x}{(\ln x)^{1-1/d}},
\end{align*}where $c_{2}=c_{2}(R,a)$ is a positive constant, depending only on $R$ and $a$. Theorem \ref{T2} is proved.

\end{proof}

\section{Acknowledgements}

The author is grateful to the anonymous referee for useful comments. The article was prepared within the framework of the Basic Research Program at HSE University, RF.

\end{document}